\documentclass[a4paper,11pt,leqno]{amsart}
\usepackage{amsmath}
\usepackage{amsthm}
\usepackage{amsfonts}
\usepackage{amssymb}
\usepackage{texdraw}

\usepackage[usenames]{color}

\newtheorem{thm}{Theorem}[section]{\bf}{\it}
\newtheorem{lem}[thm]{Lemma}{\bf}{\it}
\newtheorem{prop}[thm]{Proposition}{\bf}{\it}
\newtheorem{cor}[thm]{Corollary}{\bf}{\it}
{\bf}{\it}
{\bf}{\it} % Main Theorems, numbered A,B,...
{\bf}{\it}

\theoremstyle{definition}

{\bf}{\it}
\theoremstyle{remark}

\numberwithin{equation}{section}

\newcommand{\R}{\mathbb R}

\newcommand{\loc}{{\operatorname{loc}}}

\newcommand{\spt}{\operatorname{spt}}

\newcommand{\dist}{{\operatorname{dist}\,}}
\newcommand{\id}{{\operatorname{id}}}

\newcommand{\degree}{{\operatorname{deg}}}

\newdimen\vintkern\vintkern11pt
\def\vint{-\kern-\vintkern\int}

\newcommand{\dt}{\;\mathrm{d}t}
\newcommand{\ds}{\;\mathrm{d}s}

\newcommand{\dx}{\;\mathrm{d}x}
\newcommand{\dy}{\;\mathrm{d}y}

\newcommand{\dphi}{\;\mathrm{d}\phi}
\newcommand{\norm}[1]{\lVert #1 \rVert}

\newcommand{\cE}{\mathcal{E}}

\newcommand{\Jnorm}[1]{\nparallel #1 \rVert}

\newcommand{\haus}{\mathcal{H}}
\newcommand{\dhaus}{\;\mathrm{d}\haus}

\newcommand{\hop}{\mathcal{K}}
\newcommand{\cT}{\mathcal{T}}

\newcommand{\trace}{\mathrm{tr}}
\begin{document}

\title{Quasiconformal extension fields}
\date{}
\author{Pekka Pankka \and Kai Rajala}

\thanks{Both authors were supported 
by the Academy of Finland and P.P. by NSF grant DMS-0757732. % and by the Vilho, Yrj\"o and Kalle V\"ais\"al\"a foundation. 
Part of this research was done when P.P. was visiting University of Jyv\"askyl\"a. He wishes to thank the department for the hospitality.}  

\begin{abstract}
We consider extensions of differential fields of mappings and obtain a lower energy bound for quasiconformal extension fields in terms of the topological degree. 
We also consider the related minimization problem for the $q$-harmonic energy, and show that the energy minimizers admit higher integrability.
\end{abstract}

\maketitle

\section{Introduction}

A continuous mapping $f\colon \R^n\setminus A \to \R^n$, where $A$ is an open annulus in $\R^n$, can be extended, by classical methods, to a continuous mapping $\bar f\colon \R^n\to \R^n$. Although orientation preserving mappings do not need to admit orientation preserving extension in general, for homeomorphisms this extension problem has a solution in the form of the \emph{annulus theorem}; %the classical references are 
see e.g. \cite{KirbyR:Stahac} and \cite{KirbyR:Fouetm} for detailed discussions.

In the quasiconformal category the annulus theorem is due to Sullivan \cite{SullivanD:Hypgh} and it yields that given a quasiconformal embedding $f\colon \R^n\setminus A \to \R^n$, where $A$ is an annulus, there exists a quasiconformal mapping $\bar f \colon \R^n \to \R^n$ so that $\bar f|\R^n\setminus A'  = f$, where $\bar A\subset \mathrm{int} A'$; the distortion of the extension is quantitatively controlled. We refer to Tukia-V\"ais\"al\"a \cite{TukiaP:Lipqae} for a detailed discussion on the annulus theorem in the quasiconformal category. A simple consequence of the annulus theorem is that a mapping $f\colon \R^n\setminus A\to \R^n$, that is quasiconformal embedding in the components of $\R^n\setminus A$, can be extended to a quasiconformal mapping $\R^n\to \R^n$ if we are allowed to precompose $f$ with a Euclidean similarity in one of the components of $\R^n\setminus A$. For more general non-injective mappings of quasiconformal type, i.e., for quasiregular mappings, extension results of this type are not known. We refer to \cite{ReshetnyakY:Spambd} and \cite{RickmanS:Quam} for the theory of quasiregular mappings. In this article, we discuss quantitative estimates, in terms of the degree, for the non-existence of extensions. 

If we focus on matrix fields instead of the differential fields of mappings, it is easy to see that the extension problem admits an orientation preserving solution in the sense that the differential of an orientation preserving $C^1$-mapping $f\colon \R^n\setminus A \to \R^n$ admits an extension to a continuous matrix field $M$ on $\R^n$ having non-negative determinant. This matrix field is not, in general, a differential field of a mapping, but it can be integrated to obtain a mapping $\tilde f\colon \R^n \to \R^n$ so that the difference $\tilde f -f$ is bounded. The difference $M-D\tilde f$ can be viewed to measure the non-exactness of the extension field $M$.

We estimate the non-exactness of extensions for differential fields in the context of quasiconformal geometry, i.e. we consider matrix fields satisfying the quasiconformality condition
\begin{equation}
\label{eq:QC1}
|M(x)|^n \le K \det M(x)\quad \mathrm{a.e.\ in\ } A,
\end{equation}
where $|M(x)|$ is the operator norm of the matrix $M(x)$. Our main theorem gives a quantitative estimate for the non-exactness of the extension in terms of the degree information on the underlying mappings. For the statement, we introduce some notation. Let $B^n(r)$ be a Euclidean ball of radius $r>0$ about the origin. We denote $A(r,R)=B^n(R)\setminus \bar B^n(r)$ for $0<r<R$.

For notational convenience, we consider $1$-forms and $1$-(co)frames instead of vectors and matrix fields. We say that an $n$-tuple $\rho=(\rho_1,\ldots, \rho_n)$ of measurable $1$-forms on a domain $\Omega$ is a \emph{measurable frame}. Moreover, for $p \geq 1$ and $q \geq 1$, we say that $\rho$ is a \emph{$W_{p,q}$-frame} if $\rho_i\in L^p(\bigwedge^1 \Omega)$ and $d\rho_i \in L^q(\bigwedge^2 \Omega)$ for every $i=1,\ldots,n$. The local space $W^\loc_{p,q}$ of frames is defined similarly. 

A frame $\rho$ is said to be \emph{$K$-quasiconformal in $A(r,R)$} if  
\begin{equation}
\label{eq:QC}
\tag{QC}
|\rho|^n \le K \star (\rho_1\wedge \cdots \wedge \rho_n)\quad \mathrm{a.e.\ in\ } A(r,R),
\end{equation}
where $|\rho|$ is the operator norm of $\rho$, see Section \ref{sec:preli}. After the natural identification of frames and matrix fields, the two conditions \eqref{eq:QC1} and \eqref{eq:QC} coincide. 

Let $0<r<R$, and let $\rho_0$ and $\rho_1$ be frames defined on $B^n(r)$ and $\R^n\setminus \bar B^n(R)$, respectively. We say that a frame $\rho$ \emph{$K$-quasiconformally connects $\rho_0$ and $\rho_1$ in $A(r,R)$} if $\rho$ is $K$-quasiconformal in $A(r,R)$ and satisfies $\rho|B^n(r) = \rho_0|B^n(r)$ and $\rho|(\R^n\setminus \bar B^n(R)) = \rho_1$. In our main theorem we assume that $\rho_0$ and $\rho_1$ are (the restrictions of) $df_0$ and $dx$,  respectively, where $f_0 \colon \R^n \to \R^n$ is a continuous $W^{1,n}_\loc$-mapping and $dx$ is the standard frame $dx=(dx_1,\ldots, dx_n)$.%; $dx$ is also the differential of the identity mapping.

\begin{thm}
\label{thm:main}
Let $f_0 \in W^{1,n}_\loc(\R^n,\R^n)$ be a continuous mapping, $0<r<\infty$, $p>n$, and $n\ge 3$.  Suppose that a $W_{p,n/2}$-frame $\rho$ $K$-quasiconformally connects $df_0$ and $dx$ in $A(r/2,r)$. Then
\begin{equation}
\label{eq:main}
\int_{\R^n} \max\{\deg(y,f_0,B^n(r/2))-1,0\} \dy \le C\norm{d\rho}_{n/2}^n, 
\end{equation}
where $C=C(n,K)>0$.
\end{thm}

%We expect similar results to hold in the plane, but with the $L^{n/2}$-norm replaced by other norms. 
Similar results also hold in the plane, but with the $L^{n/2}$-norm replaced by other norms.
The estimate \eqref{eq:main} can be interpreted as a lower bound for the minimal energy of the extension frame. For the statement of our next result, let $A$ be an open annulus in $\R^n$. Given $W_{n,q}^{\loc}$-frames $\rho_0$ and $\rho_1$ in $\R^n$, we denote by $\cE_{q,K}(\rho_0,\rho_1;A)$ the set of $W_{n,q}^{\loc}$-frames $\rho$ $K$-quasiconformally connecting $\rho_0$ and $\rho_1$ in $A$. 

\begin{thm}
\label{thm:blah}
Let $q>n/2$, $n\geq 2$, $A$ an annulus in $\R^n$, and let $\rho_0$ and $\rho_1$ be $K$-quasiconformal $W_{n,q}^{\loc}$-frames in $\R^n$ that can be $K$-quasiconformally connected in $A$. Then there exists a $W_{n,q}^{\loc}$-frame $\rho\in \cE_{q,K}(\rho_0,\rho_1;A)$ so that 
\begin{equation}
\label{eq:blah}
\int_{A} |d\rho|_2^q = \inf_{\rho'} \int_{A} |d\rho'|_2^q,
\end{equation}
where the infimum is taken over $\rho' \in \cE_{q,K}(\rho_0,\rho_1;A)$, and the norm $|\cdot|_2$ is the Hilbert-Schmidt norm in $\bigwedge^2 \R^n$. Moreover, there exists $p=p(n,K)>n$ so that $\rho \in L^p_\loc(\bigwedge^1 A)$.
\end{thm}

This paper is organized as follows. In Section \ref{sec:Poinc}, we discuss the $L^p$-Poincar\'e homotopy operator $\cT$ of Iwaniec and Lutoborski. This operator plays a crucial role in both of our theorems by providing a Sobolev-Poincar\'e inequality for $W_{p,q}$-frames. The interplay between degree of the potential $\cT\rho$ and the energy of $\rho$ is then discussed in Section \ref{sec:estimate}. A continuity estimate for $\cT\rho$ is proven in Section \ref{sec:cont}, and the proof of Theorem \ref{thm:main} is given in Section \ref{sec:proof_of_1}. In Section \ref{sec:min} we consider the variational problem for the energy and prove Theorem \ref{thm:blah}.

\section{Preliminaries}
\label{sec:preli}

The open ball in $\R^n$ about $x_0$ with radius $r>0$ is denoted by $B^n(x_0,r)$. For $x_0=0$ we abbreviate $B^n(r)=B^n(0,r)$ and $B^n=B^n(1)$. The corresponding closed balls are denoted by $\bar B^n(x_0,r)$, $\bar B^n(r)$, and $\bar B^n$. The sphere of radius $r$ about the origin is denoted by $S^{n-1}(r)$ and the unit sphere in $\R^n$ by $S^{n-1}$. Given a ball $B=B^n(x,r)$ we commonly use also notation $\lambda B$ to denote the ball $B^n(x,\lambda r)$ for $\lambda>0$. 

Given a frame $\rho$, we denote by $|\rho|$ the operator norm
\[
|\rho| = \sup_{(v_1,\ldots,v_n)} |(\rho_1(v_1,\ldots, v_n), \ldots, \rho_n(v_1,\ldots, v_n))|,
\]
where the supremum is taken over $n$-tuples $(v_1,\ldots, v_n)$ satisfying $ \sum_i |v_i|^2 = 1$. We abuse the common terminology slightly and call $J_\rho = \star (\rho_1 \wedge \cdots \wedge \rho_n)$ the \emph{Jacobian} of $\rho$, although we also write $J_f=\operatorname{det}(Df)$ when $f$ is mapping. 

Let $\Omega$ be a domain in $\R^n$. The \emph{weak exterior differential} of an $\ell$-form $\omega \in L^1_\loc(\bigwedge^\ell \Omega)$ is the unique form 
$d\omega \in L^1_\loc(\bigwedge^{\ell+1} \Omega)$, if exists, that satisfies  
\[
\int_\Omega d\omega \wedge \varphi = (-1)^{\ell+1} \int_\Omega \omega \wedge d\varphi
\]
for every $\varphi\in C^\infty_0(\bigwedge^{n-\ell-1} \Omega)$. Given $1\le p < \infty$ and $1\le q <\infty$, we denote by $W_{p,q}(\bigwedge^\ell \Omega)$ \emph{the (p,q)-partial Sobolev space} of the $\ell$-forms $\omega \in L^p(\bigwedge^\ell \Omega)$ having $d\omega \in L^q(\bigwedge^{\ell+1}\Omega)$. We will also say that a measurable $\ell$-form $\omega$ in $\Omega$ belongs to the Sobolev space $W^{1,p}(\bigwedge^\ell \Omega)$ if $\omega_I\in W^{1,p}(\Omega)$, where $\omega = \sum_I \omega_I dx_I$. Here $dx_I = dx_{i_1} \wedge \cdots \wedge dx_{i_\ell}$ for $I=(i_1,\ldots, i_\ell)$.
 
We call an $n$-tuple $\rho = (\rho_1,\ldots, \rho_n)$ of (Borel) measurable $1$-forms on $\Omega$ a \emph{measurable frame}. We say that a measurable frame is a \emph{$W_{p,q}$-frame} if the forms $\rho_i$, $i=1,\ldots,n$, belong to $W_{p,q}$. We then denote 
\[
d\rho = (d\rho_1,\ldots, d\rho_n).
\]

\subsection{Topological degree}
\label{gre}

Let $f\colon \bar B^n(r)\to \R^n$ be a continuous mapping %on the sphere $S^{n-1}(r)$, 
and $y \in \R^n \setminus fS^{n-1}(r)$. Then the local degree $\degree(y,f,B^n(r))$ of $f$ at $y$ with respect to $B^n(r)$ is the mapping degree of $g \colon S^{n-1}\to S^{n-1}$, 
\[
g(x)=\frac{f(rx)-y}{|f(rx)-y|}. 
\]

If $f \colon \Omega \to \R^n$ is $C^{\infty}$, and if $G \subset \Omega$ is a domain compactly contained in $\Omega$, then the local degree satisfies the change of variables formula 
\begin{equation}
\label{cov}
\int_{G} \eta(f(x))J_f(x)\dx = \int_{\R^n} \eta(y) \degree(y,f,G)\dy 
\end{equation} 
for every non-negative $\eta \in L^1(G)$. In fact, the degree can be defined by using \eqref{cov} and the property that $\degree(y,f,G)=\degree(z,f,G)$ whenever $y$ and $z$ lie in the same component of $\R^n \setminus f(\partial \Omega)$. We will use the fact that \eqref{cov} remains valid for mappings $f \in W^{1,p}(\Omega,\R^n)$ when $p>n$; see e.g.\;\cite{MartioO:Lusin}.

We will use the following properties of the local degree; see e.g.\;\cite[I.4.2]{RickmanS:Quam}. Suppose that $f_i \colon \Omega \to \R^n$, $i=0,1$, are continuous, $G \subset \subset \Omega$ is a domain, and $y \in \R^n$. If there exists a homotopy $H:[0,1]\times \overline{G} \to \R^n$ 
so that $y \notin H([0,1]\times \partial G)$, $H(0,x)=f_0(x)$, and $H(1,x)=f_1(x)$ for every $x \in \overline{G}$, then 
\begin{equation}
\label{haka}
\degree(y,f_0,G)=\degree(y,f_1,G). 
\end{equation} 
Also, if $U \subset G$ is open, and if $y \notin f_0(\partial U \cup \partial G)$, then 
\begin{equation}
\label{arka}
\degree(y,f_0,G)=\degree(y,f_0,U)+\degree(y,f_0,G \setminus \overline{U}). 
\end{equation}

%%%%%%%%%%%%%%%%%%%%%%%%%%%%%%%%%%%%%%%%%%%%%%%%%%%%%%%%%%%%%%%%%%%%%%%%%%%%%%%
%%%%%%%%%%%%%%%%%%%%%%%%%%%%%%%%%%%%%%%%%%%%%%%%%%%%%%%%%%%%%%%%%%%%%%%%%%%%%%%
%%%%%%%%%%%%%%%%%%%%%%%%%%%%%%%%%%%%%%%%%%%%%%%%%%%%%%%%%%%%%%%%%%%%%%%%%%%%%%%

\section{Averaged Poincar\'e homotopy operator}
\label{sec:Poinc}

Iwaniec and Lutoborski introduced the $L^p$-averaged Poincar\'e homotopy operator in \cite{IwaniecT:Intenl}. 

Given $y\in \R^n$, we denote by $\hop_y \colon C^{\infty}(\bigwedge^\ell \R^n) \to C^{\infty}(\bigwedge^{\ell-1} \R^n)$, $\ell=1,\ldots, n-1$, the \emph{Poincar\'e homotopy operator (at $y$)}
\[
\hop_y\omega(x;v_1,\ldots,v_{\ell-1})= \int_0^1 t^{\ell-1}\omega(y+t(x-y);x-y,v_1,\ldots,v_{\ell-1})\dt. 
\]

As in \cite{IwaniecT:Intenl} we define \emph{an averaged Poincar\'e homotopy operator} $\cT$ as follows. Let $\varphi \in C_0^{\infty}(B^n(1/4))$ be non-negative with integral one. From now on we consider $\varphi$ to be fixed.

We set $\cT \colon L^1_\loc(\bigwedge^\ell \R^n) \to L^1_\loc(\bigwedge^{\ell-1}\R^n)$ by 
\begin{equation}
\label{eq:Tdef}
\cT \omega(x;v_1,\ldots,v_{\ell-1}) = \int_{\R^n} \varphi(y) \hop_y \omega(x;v_1,\ldots,v_{\ell-1})\dy;
\end{equation}
$\cT$ is well-defined by \cite[(4.15)]{IwaniecT:Intenl}. Both operators $\hop$ and $\cT$ satisfy a chain homotopy condition, which for $\cT$ reads as
\begin{equation}
\label{eq:cht}
\id = d\cT + \cT d.
\end{equation}

For all $p>1$, we can consider $\cT$ as a bounded operator $\cT \colon L^p(\bigwedge^\ell B^n) \to W^{1,p}(\bigwedge^{\ell-1} B^n)$; see \cite[Proposition 4.1]{IwaniecT:Intenl}. The chain homotopy condition together with the Sobolev embedding theorem then give the Sobolev-Poincar\'e inequality
\begin{equation}
\label{eq:sp}
\norm{\cT d\omega}_{p^*, B^n} \le C(n,p)\norm{d\omega}_{p,B^n},
\end{equation}
where $\omega \in W_{1,p}(\bigwedge^\ell B^n)$, $1<p<n$, $\ell \ge 1$, and $p^*=np/(n-p)$ is the Sobolev exponent; see \cite[Corollary 4.2]{IwaniecT:Intenl}.

The formula \eqref{eq:Tdef} naturally defines an averaged homotopy operator on $L^p(\bigwedge^\ell B^n(r))$ for all $r>0$ if the function $\varphi$ is scaled properly. To avoid such technicalities, we define for all $r>0$ the Poincar\'e homotopy operator by $\cT_r = \lambda_{1/r}^* \circ \cT \circ \lambda_r^*$, where $\lambda_r \colon \R^n\to \R^n$ is the mapping $\lambda_r(x) = rx$. By scale invariance of \eqref{eq:sp}, operators $\cT_r$ satisfy the same Sobolev-Poincar\'e inequality as $\cT$. Moreover, due to condition $\spt \varphi \subset B^n(1/4)$, they satisfy
\[
\cT_r df = f + c 
\]
on $B^n(r/4)$ for functions $f\in W^{1,1}(B^n(r))$. 

Given a frame $\rho$, we extend the notation $\cT \rho$ to denote the mapping $\cT \rho = (\cT \rho_1,\ldots, \cT \rho_n) \colon \Omega \to \R^n$.

We end this section with an application of the isoperimetric inequality
\begin{equation}
\label{eq:iso}
\int_{B} |J_f(x)| \dx \le C(n) \left( \int_{\partial B} |Df(y)|^{n-1} \dhaus^{n-1}(y)\right)^{n/(n-1)}
\end{equation}
in balls $B$ compactly contained in $\Omega$; see e.g. \cite[Chapter II, p.81]{ReshetnyakY:Spambd} or \cite[Section 6]{ReshetnyakY:Somgpf}.

\begin{lem}
\label{lemma:ipe}
Let $r>0$, $f_0 \colon \R^n \to \R^n$ be a mapping in $W^{1,n}_\loc$, $n\geq 3$, and let $\rho$ be a $W_{n,n/2}^\loc$-frame in $\R^n$. Then there exists $C=C(n)>0$ so that
\[
\int_{B^n(r)} |J_{\cT_r\rho}| \le C \left( \norm{\rho}_{n,A(r,2r)} + \norm{d\rho}_{n/2,B^n(r)} \right)^n. 
\]
\end{lem}

\begin{proof}
By the isoperimetric inequality \eqref{eq:iso},
\[
\int_{B^n(t)} |J_{\cT_r\rho}| \le C\left( \int_{\partial B^n(t)} |d\cT_r\rho|^{n-1}\right)^{\frac{n}{n-1}}
\]
for almost every $r\leq t \leq 2r$. Thus the Sobolev-Poincar\'e inequality \eqref{eq:sp} gives 
\begin{eqnarray*}
\int_{r}^{2r} \left( \int_{B^n(t)} |J_{\cT_r\rho}|\right)^{\frac{n-1}{n}} \dt 
&\le& C \int_r^{2r} \left( \int_{\partial B^n(t)}  |d\cT_r\rho|^{n-1} \right) \dt \\
&\le& C \int_{A(r,2r)} |d\cT_r\rho|^{n-1} \le C \int_{A(r,2r)} \left( |\rho|+ |\cT_rd\rho|\right)^{n-1} \\
&\le& C \left( \norm{\rho}_{n-1, A(r,2r)} + \norm{\cT_rd\rho}_{n-1,A(r,2r)} \right)^{n-1}  \\
&\le& C \left( r^{\frac{1}{n-1}}  \norm{\rho}_{n, A(r,2r)} + r^{\frac{1}{n-1}} \norm{d\rho}_{n/2,B^n(2r)} \right)^{n-1} \\
&\le& C r \left( \norm{\rho}_{n, A(r,2r)} + \norm{d\rho}_{n/2,B^n(2r)} \right)^{n-1}.
\end{eqnarray*}
Therefore, 
\begin{eqnarray*}
\left( \int_{B^n(r)} |J_{\cT_r\rho}|\right)^{\frac{n-1}{n}} &\le&
\frac{1}{r} \int_{r}^{2r} \left( \int_{B^n(t)}
|J_{\cT_r\rho}|\right)^{\frac{n-1}{n}} \dt \\ &\le& C \left(
\norm{\rho}_{n, A(r,2r)} + \norm{d\rho}_{n/2,B^n(2r)} \right)^{n-1}.
\end{eqnarray*}
The claim follows.
\end{proof}

%%%%%%%%%%%%%%%%%%%%%%%%%%%%%%%%%%%%%%%%%%%%%%%%%%%%%%%%%%%%%%%%%%%%%%%%%%%%%%
%%%%%%%%%%%%%%%%%%%%%%%%%%%%%%%%%%%%%%%%%%%%%%%%%%%%%%%%%%%%%%%%%%%%%%%%%%%%%%
%%%%%%%%%%%%%%%%%%%%%%%%%%%%%%%%%%%%%%%%%%%%%%%%%%%%%%%%%%%%%%%%%%%%%%%%%%%%%%

\section{Energy and local degree}
\label{sec:estimate}

In this section we prove an integral estimate which relates the degree of the mapping $\cT \rho$ and the energy of the frame $\rho$. Given a continuous mapping $f \colon \overline{A(r,R)} \to \R^n$ we denote by $\Omega_-(f)$ the set 
\[
\Omega_-(f) = \{y \in \R^n \colon \deg(y, f, A(r,R))<0\}.
\]
The main result of this section reads as follows. 

\begin{prop}
\label{prop:hius}
Let $0< r < R <\infty$, $p>n$, and $n\ge 3$. Suppose that $\rho$ is a $W_{p,n/2}^\loc$-frame on $\R^n$ such that $\rho$ is $K$-quasiconformal in $A(r,R)$. Then 
\begin{equation}
\label{aska}
- \int_{\Omega_-(\cT_R\rho)} \deg(y, \cT_R \rho, A(r,R)) \dy \le C \norm{d\rho}_{n/2}^n,
\end{equation}
where $C=C(n,K,\varphi)>0$.
\end{prop}

\begin{proof}
Set $\tilde G = (\cT_R\rho)^{-1}(\Omega_-(\cT_R\rho))\cap A(r,R)$ and
\[
I = - \int_{\Omega_-(\cT_R\rho)} \deg(y, \cT_R \rho, A(r,R)) \dy.
\]
Since $\cT_R\rho \in W^{1,p}_\loc$ with $p>n$, we have, by the change of variables \eqref{cov} and by \eqref{eq:cht}, that
\begin{eqnarray*}
- I &\ge& \int_{\tilde{G}} d\cT_R\rho_1 \wedge \ldots \wedge d\cT_R \rho_n \\
&=& \int_{\tilde{G}} (\rho_1-\cT_R d\rho_1)\wedge \ldots \wedge (\rho_n-\cT_R d\rho_n). 
\end{eqnarray*}
The last integrand can be estimated from below by 
\[
J_{\rho}(x) - C \sum_{k=1}^n |\cT_R d\rho(x)|^k|\rho(x)|^{n-k}  
\]
almost everywhere. Here $C=C(n)$.

Then, by H\"older's inequality and the Sobolev-Poincar\'e inequality \eqref{eq:sp}, we obtain 
\begin{equation}
\label{eq:SP}
\begin{split}
- I&\ge \int_{\tilde{G}} J_{\rho}(x)\dx - C_0\sum_{k=1}^n \norm{\cT_R d \rho}_{n,B^n(R)}^k \norm{\rho}_{n,\tilde{G}}^{n-k} \\
&\ge   \int_{\tilde{G}} J_{\rho}(x)\dx - C_0\sum_{k=1}^n \norm{d \rho}_{n/2,B^n(R)}^k \norm{\rho}_{n,\tilde{G}}^{n-k}, 
\end{split}
\end{equation}
where $C_0=C_0(n)\ge 1$. 

Since $\rho$ is $K$-quasiconformal, 
\[
K \int_{\tilde G} J_{\rho}(x)\dx \geq \norm{\rho}_{n,\tilde{G}}^n.
\]
Thus
\begin{equation}
\label{reiat}
\frac{1}{K} \norm{\rho}_{n,\tilde{G}}^n + I \leq C_0  \sum_{k=1}^n \norm{d \rho}_{n/2,B^n(R)}^k \norm{\rho}_{n,\tilde{G}}^{n-k}. 
\end{equation}
We show that  
\begin{equation}
\label{eq:r1}
\norm{\rho}_{n,\tilde G} \le KC_0n \norm{d\rho}_{n/2,B^n(R)}.
\end{equation}
Suppose towards contradiction that \eqref{eq:r1} does not hold. Then, by \eqref{reiat},
\begin{eqnarray*}
\norm{\rho}_{n,\tilde{G}}^n &\le& C_0K  \sum_{k=1}^n \norm{d \rho}_{n/2,B^n(R)}^k \norm{\rho}_{n,\tilde{G}}^{n-k} \\ 
&<& C_0K\sum_{k=1}^n \left(\frac{1}{C_0Kn}\right)^k \norm{\rho}_{n,\tilde G}^k \norm{\rho}_{n,\tilde G}^{n-k} \le \norm{\rho}_{n,\tilde G}^n. 
\end{eqnarray*}
This is a contradiction. Thus \eqref{eq:r1} holds.

We may now estimate $I$ using \eqref{reiat} and \eqref{eq:r1} to obtain   
\[
I \le (2KC_0n)^n \norm{d\rho}_{n/2,B^n(R)}^n. 
\]
This concludes the proof.
\end{proof}

%%%%%%%%%%%%%%%%%%%%%%%%%%%%%%%%%%%%%%%%%%%%%%%%%%%%%%%%%%%%%%%%%%%%%%%%%%%%%%%
%%%%%%%%%%%%%%%%%%%%%%%%%%%%%%%%%%%%%%%%%%%%%%%%%%%%%%%%%%%%%%%%%%%%%%%%%%%%%%%

\section{A continuity estimate}
\label{sec:cont}

The second main ingredient in the proof of Theorem \ref{thm:main} is the following continuity estimate for $\cT \rho$.

\begin{lem}
\label{lemma:cont}
Let $u\in W^{1,1}_\loc(\R^n)$, $n \geq 2$, and let $\rho$ be a $W_{1,1}^\loc$-form in $\R^n$ so that $\rho = du$ in $\R^n\setminus \bar B^n$. Then there exists $C=C(n)$ so that
\begin{equation}
\label{eq:sup_est}
|\cT \rho(x) - \cT \rho(y)-(u(x)-u(y))| \le C \norm{d\rho}_1
\end{equation}
for almost every $x$ and $y \in \R^n\setminus B^n(2)$. 
\end{lem}

For the proof of this lemma, we introduce some notation. Given points $x$ and $y$ in $\R^n$ we denote the (oriented) line segment from $x$ to $y$ by $[x,y]$. Given points $x$, $y$, and $z$ in $\R^n$ we denote by $[x,y,z]$ the (oriented) $2$-simplex that is the convex hull of $\{x,y,z\}$. Similarly, for affinely independent points, we define $L(x,y)$ and $P(x,y,z)$ to be the (unique) line and plane containing $\{x,y\}$ and $\{x,y,z\}$, respectively.

\begin{proof}
By the density of smooth frames in $W_{1,1}$, we may assume that $\rho$ is smooth. We assume that $n\geq 3$, the simpler planar case is left for the reader. 

Let $a$ and $b \in \R^n\setminus B^n(2)$. We may assume that $|a|\le |b|$ and that $L(a,b)\cap B^n = \emptyset$. Otherwise, we consider an additional point $c\in S^{n-1}(2)$ so that $\left([a,c]\cup [c,b]\right)\cap B^n=\emptyset$. Indeed, we can take $c\in S^{n-1}(2)$ so that $[0,c]$ bisects the angle between $[0,a]$ and $[0,b]$ in the plane $P(0,a,b)$. 

Then, since $\rho = du$ outside $B^n$, 
\begin{eqnarray*}
\hop_y\rho(b)-\hop_y\rho(a) &=& \int_{[y,b]} \rho - \int_{[y,a]} \rho - \int_{[a,b]} \rho +u(b)-u(a) \\
&=& \int_{[y,b,a]} d\rho + u(b)-u(a) 
\end{eqnarray*}
for all $y \in \R^n$. Thus
\begin{eqnarray*}
|\cT\rho(b)-\cT\rho(a)-(u(b)-u(a))| & \le & \int_{B^n} \varphi(y) \left( \int_{P(y,a,b)} |d\rho| \right)\dy \\
& \le & ||\varphi||_{\infty} \int_{B^n} \left( \int_{P(y,a,b)} |d\rho| \right)\dy. 
\end{eqnarray*}

Let $\psi$ be a Euclidean isometry so that $\psi(0)=a$ and $\psi L(0,e_n) = L(a,b)$. Then $B=\psi^{-1}(B^n)$ is a ball in $A(|a|-1,|a|+1)$. 

Since $\psi$ is an isometry, we have
\begin{equation}
\label{voimaa}
\int_{B^n} \left( \int_{P(y,a,b)} |d\rho| \right)\dy 
= \int_B \left( \int_{P(x,0,e_n)} |d\psi^*\rho|\right) \dx
= \int_B \Psi(x) \dx,
\end{equation}
where $\Psi \colon \R^n \to \R$ is defined by 
\[
\Psi(x) = \int_{P(x,0,e_n)} |d\psi^*\rho|.
\]

To estimate the integral in \eqref{voimaa}, we observe first that, given $x\in \R^n \setminus L(0,e_n)$, we have $\Psi(x)=\Psi(y)$ for $y \in P(x,0,e_n)$. 
Then, writing $x=(s,\phi,x_n)$ in cylindrical coordinates, we see that $\Psi(x)=\Psi(\phi)$. We denote $p=(p_1,\ldots, p_n)=\psi^{-1}(0)$. Then 
\begin{eqnarray*}
\int_B \Psi(x)\dx &\leq& C \int_{p_n-1}^{p_n+1} \int_{S^{n-2}}\int_{|a|-1}^{|a|+1}s^{n-2}\Psi(\phi) \ds \dphi \dx_n \\ 
&\leq& C|a|^{n-2}\int_{S^{n-2}} \Psi(\phi)\dphi. 
\end{eqnarray*}
We write $P(x,0,e_n)=P(\phi)$. Then, as  
$$
\Psi(\phi)=\int_{P(\phi)} |d\psi^*\rho|= \int_{-\infty}^{\infty}\int_{-\infty}^{\infty} |d\psi^*\rho(s,\phi,x_n)| \ds\dx_n, 
$$
another application of cylindrical coordinates yields 
\begin{eqnarray*}
|a|^{n-2} \int_{S^{n-2}} \Psi(\phi)\dphi &\leq& C |a|^{n-2}\int_B \frac{|d\psi^*\rho(x)|}{\dist(x,L(0,e_n))^{n-2}}\dx \\  &\leq& C \norm{d\rho}_1. 
\end{eqnarray*}
The claim follows by combining the estimates. 
\end{proof}

%%%%%%%%%%%%%%%%%%%%%%%%%%%%%%%%%%%%%%%%%%%%%%%%%%%%%%%%%%%%%%%%%%%%%%%%%%%%%%
%%%%%%%%%%%%%%%%%%%%%%%%%%%%%%%%%%%%%%%%%%%%%%%%%%%%%%%%%%%%%%%%%%%%%%%%%%%%%%

\section{Degree estimate and Proof of Theorem \ref{thm:main}}
\label{sec:proof_of_1}

\begin{lem}
\label{lem:liimaus_identtiseen} 
There exists $\varepsilon=\varepsilon(n)>0$ so that if $\rho$ is a $W_{p,n/2}$-frame, $p>n$, $n \geq 3$, satisfying $\norm{d\rho}_1 \le \varepsilon$, and if 
$\rho = dx$ in $\R^n\setminus B^n$, then
\[
\deg(y, \cT \rho, B^n(2))\le 1
\]
for all $y \in \R^n\setminus \cT\rho S^{n-1}(2)$. 
\end{lem}

\begin{proof}
By Lemma \ref{lemma:cont},
\[
|\cT\rho(x)-\cT\rho(2e_1) - (x-2e_1)| \le C \norm{d\rho}_1
\]
for $x\in S^{n-1}(2)$. Suppose from now on that $C \norm{d\rho}_1 < 1/8$.

Let $f\colon \R^n \to \R^n$ be the mapping $f(x) = \cT\rho(x) + (2e_1 - \cT\rho(2e_1))$. We denote $v=2e_1-\cT\rho(2e_1)$. Since
\[
\deg(y+v, f, B^n(2)) = \deg(y, \cT\rho, B^n(2))
\]
for all $y \not \in \cT\rho S^{n-1}(2)$, the claim of the lemma holds if and only if $\deg(z,f,B^n(2))\le 1$ for all $z\not \in fS^{n-1}(2)$.

Since
\begin{equation}
\label{eq:f_to_id}
|f(x)-x| = |\cT\rho(x)-\cT\rho(2e_1) - (x-2e_1)| < C\norm{d\rho}_1 < 1/8
\end{equation}
for $x\in S^{n-1}(2)$, we have, by a homotopy argument,
\begin{equation} 
\label{eq:f_id}
\deg(y,f,B^n(2)) = \deg(y,\id,B^n(2))
\end{equation}
for $y \not \in A(15/8,17/8)$. Moreover,
\begin{equation}
\label{eq:f_dist}
\left| |f(x)| -2 \right| < C\norm{d\rho}_1 < 1/8
\end{equation}
on $S^{n-1}(2)$. 

Suppose now that there exists $y \in A(15/8,17/8)$ so that
\begin{equation}
\label{eq:f_y}
\deg(y,f,B^n(2))\ge 2.
\end{equation}
By continuity, we can fix $r>0$ so that $\deg(y',f,B^n(2))\ge 2$ for every $y'\in B^n(y,r)$.

By density of smooth frames in $W_{p,q}$ and continuity of $\cT \colon L^p(\bigwedge^1 B^n(2)) \to W^{1,p}(B^n(2))$, we may fix a smooth frame $\tilde \rho$ so that $\tilde f = \cT \tilde \rho$ satisfies \eqref{eq:f_to_id}, and hence also \eqref{eq:f_dist}, in place of $f$ and $\deg(y',\tilde f,B^n(2))\ge 2$ for $y'\in B^n(y,r/2)$. We may also assume that $\tilde \rho = \rho = dx$ on $\R^n\setminus B^n$.

By \eqref{eq:f_to_id}, the mapping $g \colon S^{n-1} \to S^{n-1}$, 
\[
g(x) = \frac{\tilde f(2x)}{|\tilde f(2x)|},
\]
is well-defined and smooth. 

We show that $J_g \ge 0$ almost everywhere on $S^{n-1}$. This contradicts $\deg(y,f,B(2))\ge 2$ and the claim follows.
Indeed, since $g$ is homotopic to $\id \colon S^{n-1} \to S^{n-1}$, we have, by the degree theory,
\[
\int_{S^{n-1}} J_g = \deg(g) |S^{n-1}| = |S^{n-1}|. %\deg(0, \tilde f, B(2)) |S^{n-1}| = |S^{n-1}|.
\]
Since $\deg(y',\tilde f,B^n(2))\ge 2$ for $y'\in B^n(y,r/2)$, there exists a set $E\subset S^{n-1}$ of positive $\haus^{n-1}$-measure so that $\#\left(g^{-1}(z)\right) \ge  2$ for $z \in E$. Hence, by the change of variables,
\begin{eqnarray*}
\int_{S^{n-1}} |J_g| &=& \int_{S^{n-1}} N(z,g) \dhaus^{n-1}(z) \\
&=& \int_{S^{n-1}\setminus E} N(z,g)\dhaus^{n-1}(z) + \int_E N(z,g)\dhaus^{n-1}(z) \\
&>& |S^{n-1}| = \int_{S^{n-1}} J_g.
\end{eqnarray*}
This contradicts the non-negativity of $J_g$. 

It remains to show the non-negativity of $J_g$. We denote 
\[
\omega_0= \sum_{j=1}^n (-1)^{j+1} \frac{x_j}{|x|^n} dx_1\wedge \ldots \wedge \widehat{dx_j}\wedge \ldots\wedge dx_n = \sum_{j=1}^n \frac{x_j}{|x|^n} (\star dx_j). 
\]
Here $\star$ is the Hodge star operator. Then we have 
\begin{eqnarray*}
|1-J_g(x)| &=& \left|(\id^*-g^*)\omega_0\right| 
= \left| \sum_{j=1}^n \left( x_j (\star dx_j) - g^*\left(\frac{y_j}{|y|^n} (\star dy_j)\right)\right) \right|\\
&=& \left| \sum_{j=1}^n \left( x_j (\star dx_j) - (\tilde f_j/|\tilde f|^n) \tilde f^*(\star dy_j)\right) \right|\\
&\le& \sum_{j=1}^n \left| (x_j-\tilde f_j/|\tilde f|^n) \star dx_j + (\tilde f_j/|\tilde f|^n) (\star dx_j-\tilde f^*(\star dy_j)) \right| \\
&\le& n \left( M_1 + M_2 \right),
\end{eqnarray*}
where
\[
M_1 = \max_j \left|x_j-\frac{\tilde f_j}{|\tilde f|^n} \right|
\] 
and
\[
M_2 = \max_j \frac{\tilde{f}_j}{|\tilde{f}|^n}\left| \star dx_j - \tilde f^*(\star dy_j)\right|.
\] 

To estimate $M_1$ we observe that, by \eqref{eq:f_dist} and \eqref{eq:f_to_id},
\[
\left| x_j - \frac{\tilde f_j}{|\tilde f|^n}\right| \le \left| 1 - \frac{1}{|\tilde f|^n}\right| |x_j| + \frac{|x_j - \tilde f_j|}{|\tilde f|^n} \le C \norm{d\rho}_1 
\]
on $S^{n-1}(2)$, where $C=C(n)$.

To estimate $M_2$, we observe first that, on $\R^n\setminus B^n$, we have
\begin{eqnarray*}
\left| \star dx_j - \tilde f^*(\star dy_j) \right| &\le& |dx - d\tilde f| =  |dx - d\cT \tilde \rho| 
= | dx - \tilde \rho + \cT d\tilde \rho| \\
&=& |\cT d\tilde \rho|.
\end{eqnarray*}
Since
\begin{eqnarray*}
|\cT d \tilde\rho(x)|&=&\left|\int_{B^n} \varphi(y)\hop_y d\tilde\rho(x)\dy\right|  \\
&\le& C \int_{S^{n-1}(x,4)} \int_0^1 |d\tilde\rho(y+t(x-y))| \dt \dhaus^{n-1}(y) \\
&\le& C \int_{\R^n} \frac{|d\tilde\rho(y)|}{|x-y|^{n-1}} \\ 
&=& C\int_{B^n} \frac{|d\tilde\rho(y)|}{|x-y|^{n-1}}\dy \le C\norm{d\tilde\rho}_1
\end{eqnarray*}
for $x\in S^{n-1}(2)$, we have that
\[
M_2 \le C \norm{d\rho}_1,
\]
where $C=C(n)$.
We choose $\varepsilon=\varepsilon(n)>0$ so that $M_1 + M_2 < 1/(2n)$ for $\norm{d\rho}_1 < \varepsilon$. Then $J_g > 1/2$. The claim follows.
\end{proof}

\begin{proof}[Proof of Theorem \ref{thm:main}]
Suppose first that $r=1$. For brevity, we denote $f = \cT \rho \colon \R^n \to \R^n$. 

Since $f = \cT\rho = f_0 + c$, where $c\in \R^n$, on $B^n(1/2)$, we have $\deg(y,f,B^n(1/2)) = \deg(y-c,f_0, B^n(1/2))$ for $y \not \in fS^{n-1}(1/2)$. Then, by Lemma \ref{lemma:ipe},
\begin{equation}
%\label{eq:loppu}
\begin{split}
& \int_{\R^n} \max\{ \deg(y,f_0,B^n(1/2)) - 1, 0\} \dy \\
&\qquad = \int_{\R^n} \max\{ \deg(y,f,B^n(1/2)) -1,0\} \dy \\
&\qquad \le \int_{B^n(1/2)} |J_f| %\le \int_{B^n} |J_f| 
\le C \left( 1 + \norm{d\rho}_{n/2}\right)^n,
\end{split}
\end{equation}
where $C=C(n)>0$.

If $\norm{d\rho}_{n/2} \ge \varepsilon$, where $\varepsilon=\varepsilon(n,K)$ is the constant in Lemma \ref{lem:liimaus_identtiseen}, the claim follows. Thus we may assume that $\norm{d\rho}_{n/2} < \varepsilon$. Since, for every $\varepsilon>0$, the $n$-measure $|fS^{n-1}(2(1+t))|=|fS^{n-1}((1+t)/2)|=0$ for almost every $t\in (-\varepsilon,\varepsilon)$,  we may assume that $|fS^{n-1}(2)|=|fS^{n-1}(1/2)|=0$ by applying a rescaling to $\rho$ if necessary.

Since
\[
\deg(y,f,A(1/2,2)) = \deg(y,f,B^n(2)) - \deg(y,f,B^n(1/2))
\]
for $y \not \in fS^{n-1}(1/2)\cup fS^{n-1}(2)$, we have, by Lemma \ref{lem:liimaus_identtiseen}, $\deg(y,f,B^n(2))\le 1$ for all $y\in \R^n\setminus fS^{n-1}(2)$. 
Thus, by Proposition \ref{prop:hius}, 
\begin{equation}
%\label{eq:loppu}
\begin{split}
& \int_{\R^n} \max\{ \deg(y,f_0,B^n(1/2)) - 1, 0\} \dy \\
&\qquad = \int_{\R^n} \max\{ \deg(y,f,B^n(1/2)) -1,0\} \dy \\
&\qquad \le - \int_{\Omega_-} \deg(y,f,A(1/2,2))\dy \le C\norm{d\rho}^n_{n/2}, 
\end{split}
\end{equation}
where $C=C(n)>0$ and $\Omega_- = \{ y\in \R^n\colon \deg(y,f,A(1/2,2))<0\}$, as in Section \ref{sec:estimate}.

For general $r>0$ the argument above can be applied to $\rho' = \left( \lambda_r^*\rho\right)/r$ and $f_0' = (f_0 \circ \lambda_r)/r$. The proof is complete.
\end{proof}

%%%%%%%%%%%%%%%%%%%%%%%%%%%%%%%%%%%%%%%%%%%%%%%%%%%%%%%%%%%%%%%%%%%%%%%%%%%%%%%
%%%%%%%%%%%%%%%%%%%%%%%%%%%%%%%%%%%%%%%%%%%%%%%%%%%%%%%%%%%%%%%%%%%%%%%%%%%%%%%
%%%%%%%%%%%%%%%%%%%%%%%%%%%%%%%%%%%%%%%%%%%%%%%%%%%%%%%%%%%%%%%%%%%%%%%%%%%%%%%

%%%%%%%%%%%%%%%%%%%%%%%%%%%%%%%%%%%%%%%%%%%%%%%%%%%%%%%%%%%%%%%%%%%%%%%%%%%%%%%
%%%%%%%%%%%%%%%%%%%%%%%%%%%%%%%%%%%%%%%%%%%%%%%%%%%%%%%%%%%%%%%%%%%%%%%%%%%%%%%
%%%%%%%%%%%%%%%%%%%%%%%%%%%%%%%%%%%%%%%%%%%%%%%%%%%%%%%%%%%%%%%%%%%%%%%%%%%%%%%

\section{Quasiconformal energy minimizers}
\label{sec:min}

In this section we consider the minimization problem for the $q$-energy of extension frames. We obtain Theorem \ref{thm:blah} in two parts. The existence of minimizers is shown in Theorem \ref{thm:eu}. For the higher integrability of minimizers, we derive an Euler-Lagrange equation (Lemma \ref{lemma:EL}) and a Caccioppoli type inequality (Corollary \ref{cor:Caccioppoli}) for this variational problem. We then establish a reverse H\"older inequality (Theorem \ref{thm:HI}) which yields the higher integrability by Gehring's lemma.

We denote by $\langle \cdot, \cdot \rangle$ the inner product
\[
\langle X,Y \rangle = \frac{1}{m} \trace \left( X^t Y \right)
\]
for $(m\times k)$-matrices and by $|\cdot|_2$ the (normalized) Hilbert-Schmidt norm 
\[
|X|_2^2 = \langle X,X \rangle = \frac{1}{n} \sum_{i=1}^m \sum_{j=1}^k X_{ij}^2.
\]
In this section, we identify frames with matrix fields, and use the inner product $\langle \cdot, \cdot \rangle$ and the norm $|\cdot|_2$ also for frames.

Let $q>n/2$ and $0<r<R<\infty$. Let $\rho_0$ and $\rho_1$ be $W_{n,q}^{\loc}$-frames in $B^n(r)$ and $\R^n\setminus B^n(R)$, respectively. The following lemma shows that $\rho_0$ and $\rho_1$ can be quasiconformally connected if they have quasiconformal extensions to the neighborhoods of $S^{n-1}(r)$ and $S^{n-1}(R)$, respectively.

\begin{lem}
\label{lemma:competitors}
Let $0<r<r'<R'<R<\infty$ and let $\rho_0$ and $\rho_1$ be $W_{n,q}^{\loc}$-frames in $B^n(r')$ and $\R^n \setminus B^n(R')$, respectively, so that $\rho_0$ is $K$-quasiconformal in $A(r,r')$ and $\rho_1$ is $K$-quasiconformal in $A(R',R)$. Then there exists a $W_{n,q}$-frame $\rho$ so that $\rho\in \cE_{q,\tilde K}(\rho_0,\rho_1;A(r,R))$, where $\tilde K = \tilde K(n,K,r,r',R,R')$.
\end{lem}

\begin{proof}
Let $r_0 = (r'+R')/2$. We define mappings $\lambda_0 \colon A(r,r_0) \to A(r,r')$ and $\lambda_1 \colon A(r_0,R) \to A(R',R)$ by
\[
\lambda_0(x)= \left( \frac{r'-r}{r_0-r}(|x|-r) + r\right) \frac{x}{|x|}
\]
and 
\[
\lambda_1(x) = \left( \frac{R-R'}{R-r_0}(|x|-r_0) +R'\right)\frac{x}{|x|}.
\]
Let also $\theta \colon [0,\infty) \to [0,1]$ be a smooth function so that $\theta(t)=1$ for $t<r'$ and $t>R'$, $\theta(t) = 0$ in a neighborhood of $r_0$, and that $|d\theta|\le 3/(R'-r')$. We set $\rho$ to be the frame
\[
\rho(x) = \left\{ \begin{array}{ll}
\rho_0, & x\in B^n(r) \\
\theta(|x|) (\lambda_0^* \rho_0)(x), & x \in A(r,r_0) \\
\theta(|x|) (\lambda_1^* \rho_1)(x), & x \in A(r_0,R) \\
\rho_1, & x\in \R^n\setminus B^n(R).
\end{array}\right.
\]
Since $\rho_0$ and $\rho_1$ are $K$-quasiconformal in $A(r,r')$ and $A(R',R)$, respectively, frames $\lambda_i^* \rho_i$ are $\tilde K$-quasiconformal for $\tilde K = \tilde K(n,K,r,r',R,R')$ for $i=0,1$. Since $|d\rho(x)| \le |d\theta(|x|)| |\lambda_i^* \rho_i(x)|+ |\theta(x)||\lambda_i^*d\rho_i(x)|$ for $i=0,1$ in $A(r,r_0)$ and $A(r_0,R)$, respectively, we have that $d\rho \in L^q(\bigwedge^2 \R^n)$. Thus $\rho$ is a $W_{n,q}$-frame.  
\end{proof}

In what follows, we assume that $\cE_{q,K}(\rho_0,\rho_1;A(r,R))$ is non-empty and we consider the minimization problem
\begin{equation}
\label{eq:min_2}
I_{q,K}(\rho_0,\rho_1,A(r,R)) = \inf_{\rho \in \cE_{q,K}(\rho_0,\rho_1;A(r,R))} \int_{A(r,R)} |d\rho|_2^q, 
\end{equation}
$q>n/2$. In the forthcoming discussion, we use the observation that for every frame $\rho \in \cE_{q,K}(\rho_0,\rho_1;A(r,R))$ there exists an affine subspace of frames conformally equivalent to $\rho$; more precisely, $(1+h)\rho \in \cE_{q,K}(\rho_0,\rho_1;A(r,R))$ for all $\rho\in \cE_{q,K}(\rho_0,\rho_1;A(r,R))$ and all $h \in C^\infty_0(A(r,R))$ satisfying $h \ge -1$.

\begin{thm}
\label{thm:eu}
The minimization problem \eqref{eq:min_2} admits a minimizer $\rho \in \cE_{q,K}(\rho_0,\rho_1;A(r,R))$, i.e.,\,there exists $\rho \in \cE_{q,K}(\rho_0,\rho_1;A(r,R))$ so that 
\[
\int_{A(r,R)} |d\rho|_2^q = I_{q,K}(\rho_0,\rho_1,A(r,R)).
\]
\end{thm}

To this end, we would like to note that, since the minimization problem is considered in $\cE_{q,K}(\rho_0,\rho_1;A(r,R))$, standard convexity arguments are not at our disposal and the uniqueness of the minimizer is not guaranteed. 

We begin the proof of Theorem \ref{thm:eu} with the following lemma. We assume in what follows that $0<r<R<\infty$.
\begin{lem}
\label{lemma:rho_bound}
Let $q>n/2$, $n \geq 2$, and let $\rho$ be a $W_{n,q}^\loc$-frame in $\R^n$ so that $\rho$ is $K$-quasiconformal in $A(r,R)$. Then 
\[
\norm{\rho}_{n,A(r,R)} \le C \left( \norm{\rho}_{n,A(R,2R)} + \norm{d\rho}_{q,B^n(2R)} \right),
\]
where $C=C(n,K,q,R)>0$.
\end{lem}

\begin{proof}
Let $A=A(R,2R)$ and $B=B^n(2R)$. We set 
\[
\omega = \sum_{j=1}^n (-1)^j \cT_{2R}\rho_j\; d\cT_{2R}\rho_1 \wedge \cdots \wedge \widehat{d\cT_{2R}\rho_j} \wedge \cdots \wedge d\cT_{2R}\rho_n.
\]

As in the proof of Proposition \ref{prop:hius}, we obtain
\[
\int_{B^n(t)} d\cT_{2R}\rho_1 \wedge \cdots \wedge d\cT_{2R}\rho_n \ge \frac{1}{K}\norm{\rho}_{n,A(r,R)}^n  - C\sum_{k=1}^n \norm{\cT_{2R} d\rho}_{n,B}^k \norm{\rho}_{n,B}^{n-k}
\]
for $R\le t \le 2R$, where $C=C(n)$.

On the other hand,
\begin{eqnarray*}
\int_R^{2R} \left( \int_{S^{n-1}(t)} \omega \right) \dt &=& \int_R^{2R} \left( \int_{B^n(t)} d\omega \right) \dt \\
&=& \int_R^{2R} \left( \int_{B^n(t)} d\cT_{2R}\rho_1 \wedge \cdots \wedge d\cT_{2R}\rho_n \right) \dt
\end{eqnarray*}
and
\begin{eqnarray*}
\int_R^{2R} \left( \int_{S^{n-1}(t)} \omega \right) \dt &\le& \int_{A} |\omega| \le n \int_{A} |\cT_{2R}\rho| |d\cT_{2R}\rho|^{n-1} \\
&\le& n \norm{\cT_{2R}\rho}_{n,A} \norm{d\cT_{2R}\rho}_{n,A}^{n-1}.
\end{eqnarray*}
Thus
\[
R \norm{\rho}_{n,A(r,R)}^n \left( 1 - C \sum_{k=1}^n \frac{\norm{\cT_{2R} d\rho}_{n,B}^k}{\norm{\rho}_{n,B}^k} \right) \le K n  \norm{\cT_{2R}\rho}_{n,A} \norm{d\cT_{2R} \rho}_{n,A}^{n-1}.
\]

There exists $\varepsilon=\varepsilon(n)>0$ so that either 
\begin{equation}
\label{eq:eps_1}
\varepsilon \norm{\rho}_{n,B} \le \norm{\cT_{2R} d\rho}_{n,B}
\end{equation}
or
\begin{equation}
\label{eq:eps_2}
R \norm{\rho}_{n,A(r,R)}^n \le C \norm{\cT_{2R}\rho}_{n,A} \norm{d\cT_{2R}\rho}_{n,A}^{n-1},
\end{equation}
where $C=C(n,K)$. 

Suppose first that \eqref{eq:eps_1} holds. Then
\[
\norm{\rho}_{n,A(r,R)} \le \norm{\rho}_{n,B} \le \frac{1}{\varepsilon} \norm{\cT_{2R} d\rho}_{n,B} \le C \norm{\cT_{2R} d\rho}_{q^*,B} \le C \norm{d\rho}_{q,B},
\]
where $C=C(n, q, R)$. Here we used the Sobolev-Poincar\'e inequality \eqref{eq:sp}. 

Suppose now that \eqref{eq:eps_2} holds. The Sobolev-Poincar\'e inequality applies to the mapping $\cT_{2R}\rho$ in $A$, so 
\[
\norm{\cT_{2R} \rho}_{n,A} \le C \norm{d\cT_{2R} \rho}_{n,A}. 
\]
Therefore, another application of \eqref{eq:sp} gives 
\begin{eqnarray*}
\norm{\rho}_{n,A(r,R)} &\le& C \norm{d\cT_{2R}\rho}_{n,A} \le C\left( \norm{\rho}_{n,A}+\norm{\cT_{2R} d\rho}_{n,A}\right) \\
&\le& C \left( \norm{\rho}_{n,A} + \norm{d\rho}_{q,B}\right),
\end{eqnarray*}
where $C=C(n,K,q,R)$. 
\end{proof}

Having Lemma \ref{lemma:rho_bound} at our disposal, the standard methods in non-linear potential theory can be used to prove Theorem \ref{thm:eu}; see \cite[Chapter 5]{HeinonenJ:Nonptd}.

\begin{proof}[Proof of Theorem \ref{thm:eu}]
Suppose $(\tilde \rho_k)$ is a minimizing sequence for \eqref{eq:min_2}. %in $\cE_{q,K}(\rho_0,\rho_1,A(r,R))$ for \eqref{eq:min_2}. 
Then $(d\tilde\rho_k)$ is a bounded sequence in $L^q(\bigwedge^2 \R^n)$. Since $\tilde\rho_k$ coincides with $\rho_0$ in $B^n(r)$ and with $\rho_1$ in $\R^n\setminus B^n(R)$ for every $k$, we have, by Lemma \ref{lemma:rho_bound}, that $(\tilde \rho_k)$ is a bounded sequence in $L^n(\bigwedge^1 B^n(R))$. By passing to a subsequence if necessary, we may assume that $\tilde \rho_k \to \tilde \rho_\infty$ weakly and $d\tilde \rho_k \to d\tilde \rho_\infty$ weakly as $k \to \infty$, where $\tilde \rho_\infty\in L^n(\bigwedge^1 B^n(R))$ with $d \tilde \rho_\infty \in L^q(\bigwedge^2 B^n(R))$. By the weak lower semi-continuity of norms, we obtain
\[
\norm{d\tilde \rho_\infty}_{q,A(r,R)} = I_{q,K}(\rho_0,\rho_1,A(r,R)).
\]

Since $q>n/2$, the $K$-quasiconformality of $\tilde \rho_\infty$ is a consequence of compensated compactness \cite[Theorem 5.1]{IwaniecT:Intenl}; see also \cite[Proposition 4.5]{QCF}.

Finally, the boundary conditions $\tilde \rho_\infty |B^n(r) = \rho_0$ and $\tilde \rho_\infty |\R^n\setminus B^n(R) = \rho_1$ follow from weak convergence of the sequence $(\tilde \rho_k)$ to $\tilde \rho_\infty$. Thus $\tilde \rho_\infty \in \cE_{q,K}(\rho_0,\rho_1;A(r,R))$. 
\end{proof}

Following the standard arguments in the elliptic theory we can show that minimizers satisfy an Euler-Lagrange equation; we refer to \cite[5.13]{HeinonenJ:Nonptd} for details.

\begin{lem}
\label{lemma:EL}
A minimizer $\rho$ of the problem \eqref{eq:min_2} satisfies the equation
\begin{equation}
\label{eq:EL}
\int_{A(r,R)} \left\langle |d\rho|_2^{q-2}d\rho, d(h\rho)\right\rangle = 0
\end{equation}
for every $h\in C^\infty_0(A(r,R))$.
\end{lem}

Having the Euler-Lagrange equation at our disposal, we find an Euler-Lagrange equation for the minimizers of \eqref{eq:min_2}. 

\begin{lem}
\label{lemma:Caccioppoli}
Let $\rho$ be a minimizer of the problem \eqref{eq:min_2}. Then
\begin{equation}
\label{eq:Caccioppoli}
\left( \int_{A(r,R)} |d\rho|_2^q h^q \right)^{1/q}  \le q \left( \int_{A(r,R)} |\rho|_2^q |dh|^q \right)^{1/q}
\end{equation}
for every non-negative $h \in C^\infty_0(A(r,R))$.
\end{lem}

\begin{proof}
By the Euler-Lagrange equation \eqref{eq:EL},
\begin{eqnarray*}
0 &=& \int_{A(r,R)} \langle |d\rho|_2^{q-2}d\rho, d(h^q \rho) \rangle \\
&=& \int_{A(r,R)} \langle |d\rho|_2^{q-2}d\rho, q h^{q-1} dh \wedge \rho + h^q d\rho \rangle.
\end{eqnarray*}
Thus
\begin{eqnarray*}
\int_{A(r,R)} h^q |d\rho|_2^q &\le& q \int_{A(r,R)} |d\rho|_2^{q-1} h^{q-1} |dh| |\rho| \\
&\le& q \left( \int_{A(r,R)} |d\rho|_2^q h^q \right)^{(q-1)/q} \left( \int_{A(r,R)} |dh|^q |\rho|_2^q \right)^{1/q}.
\end{eqnarray*}   
The claim follows.
\end{proof}

Caccioppoli's inequality \eqref{eq:Caccioppoli} readily yields the following corollary.
\begin{cor}
\label{cor:Caccioppoli}
Let $B=B^n(x_0,s)$ be a ball so that $2\bar B \subset A(r,R)$. Then
\[
\Jnorm{d\rho}_{q,B} \le \frac{2q}{s} \Jnorm{\rho}_{q,2B}.
\]
\end{cor}
Here, and in what follows, we denote the integral average 
\[
\Jnorm{\omega}_{p,\Omega} = \left( \vint_\Omega |\omega|_2^p \right)^{1/p} 
\]
whenever $\Omega$ is a bounded domain in $\R^n$ and $\omega$ is an $n$-tuple of forms in $\Omega$.

The main result in this section is the following reverse H\"older's inequality.
\begin{thm}
\label{thm:HI}
Let $\rho_0$ be a minimizer of the problem \eqref{eq:min_2}. Then there exists $C=C(n)>0$ so that
\begin{equation}
\label{eq:RH}
\Jnorm{\rho_0}_{n,B}^n \le C\Jnorm{\rho_0}_{\max\{n-1,q\},2B}^n 
\end{equation}
for balls $B=B^n(x_0,s)$ satisfying $2\bar B\subset A(r,R)$.
\end{thm}

Gehring's lemma now yields the higher integrability of $\rho_0$; see e.g. \cite[Corollary 14.3.1]{IwaniecT:Geoftn}.
\begin{cor}
Let $\rho_0$ be a minimizer of the problem \eqref{eq:min_2}. Then there exists $p>n$ and $C_0=C_0(p,n)>0$ so that 
\[
\Jnorm{\rho_0}_{p,B} \le C_0 \Jnorm{\rho_0}_{n,2B}
\]
whenever $B=B^n(x_0,s)$ is a ball satisfying $2\bar B\subset A(r,R)$.
\end{cor}

\begin{proof}[Proof of Theorem \ref{thm:HI}]
For the purpose of this proof, we define $\cT_{x_0}$ to be the averaged Poincar\'e homotopy operator centered at $x_0$, that is, $\cT_{x_0} = (\tau^{-1})^* \circ \cT_{2s} \circ \tau^*$, where $\tau$ is the translation $x\mapsto x + x_0$. Naturally, the properties of $\cT$ discussed in Section \ref{sec:Poinc} hold also for $\cT_{x_0}$.

Let $\rho = (\rho_1,\ldots,\rho_n)$. By quasiconformality of $\rho$ in $A(r,R)$, we obtain, as in the proof of Proposition \ref{prop:hius}, that
\begin{eqnarray*}
\norm{\rho}_{n,B}^n &\le& C \norm{J_\rho}_{1,B} 
\le C \norm{J_{\cT_{x_0} \rho}}_{1,B} + C \sum_{k=0}^{n-1} \norm{\cT_{x_0} d\rho}_{n,B}^k \norm{\rho}_{n,B}^{n-k} \\
&=& I_1 + I_2,
\end{eqnarray*}
where $C=C(n,K)$. We estimate the integral $I_1$ first. Since $\cT_{x_0}\rho\in W^{1,n}_\loc(2B),\R^n)$, we have by the isoperimetric inequality \eqref{eq:iso},
\[
\int_{B^n(x_0,t)} |d\cT_{x_0} \rho_1\wedge \cdots \wedge d\cT_{x_0} \rho_n| \le C \left( \int_{S^{n-1}(x_0,t)} |d\cT_{x_0}\rho|^{n-1} \right)^{n/(n-1)}
\]
for almost every $r \le t \le 2r$. Thus
\begin{eqnarray*}
\left( \int_{B} |d\cT_{x_0} \rho_1\wedge \cdots \wedge d\cT_{x_0} \rho_n| \right)^{(n-1)/n} &\le& C\vint_r^{2r} \int_{S^{n-1}(x_0,t)} |d\cT_{x_0}\rho|^{n-1} \\
&\le& \frac{C}{r} \int_{2B} |d\cT_{x_0}\rho|^{n-1}.
\end{eqnarray*} 
Since
\[
\norm{d\cT_{x_0} \rho}_{n-1,2B} \le \norm{\rho}_{n-1,2B} + \norm{\cT_{x_0} d\rho}_{n-1,2B},
\]
we have, by the Sobolev-Poincar\'e and Caccioppoli's inequality,
\begin{eqnarray*}
I_1^{1/n} &\le& C r^{-1/(n-1)} \norm{\rho}_{n-1,2B} + Cr \Jnorm{\cT_{x_0} d\rho}_{n-1,2B} \\
&\le& Cr^{-1/(n-1)} \norm{\rho}_{n-1,2B} + Cr^2 \Jnorm{d\rho}_{q,2B} \\
&\le& Cr \Jnorm{\rho}_{n-1,2B} + Cr\Jnorm{\rho}_{q,2B} \\
&\le& Cr \Jnorm{\rho}_{\max\{n-1,q\},2B}.
\end{eqnarray*}
To estimate $I_2$ we use first the Sobolev-Poincar\'e inequality and then Caccioppoli's inequality to obtain
\begin{eqnarray*}
I_2 &\le& C \sum_{k=1}^n r^k \Jnorm{\cT_{x_0} d\rho}_{n,B}^k \norm{\rho}_{n,B}^{n-k} \le C \sum_{k=1}^n r^{2k} \Jnorm{d\rho}_{q,B}^k \norm{\rho}_{n,B}^{n-k} \\
&\le& C \sum_{k=1}^n r^k \Jnorm{\rho}_{q,2B}^k \norm{\rho}_{n,B}^{n-k} 
\le C r^n \sum_{k=1}^n \Jnorm{\rho}_{\max\{n-1,q\},2B}^k \Jnorm{\rho}_{n,B}^{n-k}.
\end{eqnarray*}
Combining estimates for $I_1$ and $I_2$ we have
\begin{equation}
\label{eq:4at}
\begin{split}
\Jnorm{\rho}_{n,B}^n & \le C r^{-n} ( I_1 + I_2) \\
& \le C \Jnorm{\rho}_{n-1,2B}^n + C \sum_{k=1}^n \Jnorm{\rho}_{\max\{n-1,q\},2B}^k \Jnorm{\rho}_{n,B}^{n-k},
\end{split}
\end{equation}
where $C=C(n,K,\varphi)$.

Suppose that \eqref{eq:RH} does not hold with $C_0=1/(2+2nC)$. Then, by \eqref{eq:4at},
\begin{eqnarray*}
\Jnorm{\rho}_{n,B}^n &\le& C \Jnorm{\rho}_{n-1,2B}^n + (1/2) \Jnorm{\rho}_{n,B}^n. 
\end{eqnarray*}
and \eqref{eq:4at} holds with $C_0 = 2C$. The proof is complete.
\end{proof}

\bibliographystyle{abbrv}
\bibliography{QCExtensionFields}

\vskip 5pt

\noindent P.P.
%University of Helsinki\\
Department of Mathematics and Statistics (P.O. Box 68),
FI-00014 University of Helsinki, Finland.
e-mail: pekka.pankka@helsinki.fi

\vskip 3pt

\noindent K.R.
%University of Jyv\"askyl\"a\\
Department of Mathematics and Statistics (P.O. Box 35),
FI-40014 University of Jyv\"askyl\"a, Finland.
e-mail: kai.i.rajala@jyu.fi

\end{document}